\documentclass[12pt]{amsart}
 \font\tenmsb=msbm10 scaled\magstep
1
\newfam\msbfam
   \textfont\msbfam=\tenmsb

\usepackage{amsmath}
\usepackage{amssymb}
\usepackage{mathscinet}
\usepackage{amsfonts,amsmath}
\usepackage{epsfig}
\usepackage{amsthm}
\usepackage{latexsym}
\usepackage[latin1]{inputenc}
\usepackage{verbatim}
\usepackage{pb-diagram}
\usepackage[labelformat=empty]{subfig}
\usepackage{everysel, keyval, ragged2e}
\usepackage{wrapfig}
\usepackage{color}

\newcommand{\captionfonts}{\tiny}

\makeatletter  
\long\def\@makecaption#1#2{%
  \vskip\abovecaptionskip
  \sbox\@tempboxa{{\captionfonts #1: #2}}%
  \ifdim \wd\@tempboxa >\hsize
    {\captionfonts #1: #2\par}
  \else
    \hbox to\hsize{\hfil\box\@tempboxa\hfil}%
  \fi
  \vskip\belowcaptionskip}
\makeatother   

\newtheorem{theorem}{Theorem}[section]

\theoremstyle{definition}
\newtheorem{definition}[theorem]{Definition}
\newtheorem{example}[theorem]{Example}

\theoremstyle{remark}
\newtheorem{remark}[theorem]{Remark}

\numberwithin{equation}{section}

\newcommand{\ip}[2]{\langle #1 , #2 \rangle}    
\newcommand{\bra}[1]{\langle #1 |}    
\newcommand{\ket}[1]{| #1 \rangle}    

\newcommand{\m}{\mathcal }

\newcommand{\tr}{\operatorname{Tr} }

\def\C{\mathbb C}    
\def\rank{\operatorname{rank}}   

\begin{document}

\title{Private Quantum Codes: Introduction and Connection with Higher Rank Numerical Ranges}

\author[D.~W.~Kribs and S.~Plosker]{David W.\ Kribs$^{*}$$^{\ddagger}$ and Sarah Plosker$^{*}$$^{\S}$\address{$^\S$Corresponding author. Email: ploskers@brandonu.ca}
\\\vspace{6pt} $^{*}${\em{Department of Mathematics \& Statistics, University of Guelph, Guelph, ON, Canada N1G 2W1}};$^{\ddagger}${\em{Institute for Quantum Computing, University of Waterloo, Waterloo, ON, Canada N2L 3G1}} }

\maketitle

\begin{abstract}
We give a brief introduction to private quantum codes, a basic notion in quantum cryptography and key distribution. Private code states are characterized by indistinguishability of their output states under the action of a quantum channel, and we show that higher rank numerical ranges can be used to describe them. We also show how this description arises naturally via conjugate channels and the bridge between quantum error correction and cryptography. \bigskip

\bigskip

\end{abstract}

\section{Introduction}

Private quantum codes were introduced as the quantum analogue of the classical one-time pad \cite{AMTW00,BR03}. They can be viewed from an operator theoretic perspective as input states that are indistinguishable under the action of a completely positive trace preserving map, or quantum channel \cite{CKPP11}. On the other hand, higher rank numerical ranges \cite{CKZ06a,CKZ06b,CGHK08,GLPS11, LNPST11, LP11,LPS08,LPS09,  Woe08} have been heavily studied recently with initial motivation coming from quantum error correction. 

In this paper we give a brief introduction to private quantum codes and show how they can be described in terms of higher rank numerical ranges. We also discuss connections with quantum error correction that naturally arise through the framework of conjugate quantum channels.

This paper is organized as follows. Section two discusses private quantum codes. The third section makes the connection with higher rank numerical ranges. Section four considers the Stinespring dilation theorem and some consequences in this setting, including the idea of conjugate channels and complementarity of quantum codes. We first review basic notation below.

We will restrict our attention to finite-dimensional  Hilbert spaces $\mathcal H_A$, $\mathcal H_B$ , which correspond to quantum  systems $A$ and $B$. The set of trace class operators  on $\mathcal H$ are denoted by  $\mathcal B(\mathcal H)_t$   and the set of bounded operators by $\mathcal B(\mathcal H)$. In finite dimensions, these sets coincide, and so, unless we want to draw specific attention to which sets we're working with,  we'll simply write $\mathcal L(\mathcal H)$ for both sets. We will write $X, Y$ for operators in $\mathcal B(\mathcal H)$, and $\rho, \sigma$ for density operators in $\mathcal B(\mathcal H)_t$.

Given a linear map $\Phi: \mathcal B(\mathcal H_A)_t \rightarrow \mathcal B(\mathcal H_B)_t$, its dual map $\Phi^\dagger: \mathcal B(\mathcal H_B) \rightarrow \mathcal B(\mathcal H_A)$ is defined via the Hilbert-Schmidt inner product: it is the unique map $\Phi^\dagger$ satisfying $\tr (\rho\, \Phi^\dagger(X) ) = \tr (\Phi(\rho) X)$ for all $X\in \mathcal B(\mathcal H_B)$ and all $\rho\in \mathcal B(\mathcal H_A)_t$.  Quantum channels are described by completely positive trace preserving linear (CPTP) maps. The dual of a CPTP map is a unital completely positive linear (UCP) map. The Kraus operators of a channel $\Phi$ are the operators $\{V_i\}$ given by $\Phi(\rho)=\sum_iV_i\rho V_i^*\,\, \forall \rho$. This decomposition of $\Phi$ is not unique, however, in general, results do not depend on the choice of Kraus operators.
				
We will use $^*$ to denote the complex conjugate transpose of an operator's matrix representation. We will use Dirac (bra-ket) notation: a unit column vector in $\m H$ will be denoted $\ket{\psi}$, its dual (row) vector $\ket{\psi}^*$ will be denoted $\bra{\psi}$, and the rank-one projection associated to $\ket{\psi}$ is its outer product $\ket{\psi}\bra{\psi}$. A mixed state is a convex combination of rank-one projections. We call mixed states and outer products of pure states \emph{density operators}, which are precisely the trace-one positive operators.

\section{Private Quantum Codes}

The starting point for private quantum codes is typically presented as follows \cite{AMTW00, BR03}: Alice wishes to send a quantum state to Bob without an eavesdropper, Eve, being able to learn any information about the state. A set of keys $\{1, \dots, N\}$, together with a set of unitaries $\{U_i\}$ and a probability distribution $\{p_i\}$, are all shared publicly. Each key $i\in \{1, \dots, N\}$ corresponds to the encoding $\rho\mapsto U_i\rho U_i^*$---an event which occurs with probability $p_i$.   Alice and Bob share a key, $i_0$, privately and Alice applies $U_{i_0}$ to her message. Bob receives the output message $\rho_0$, and, knowing $i_0$, can undo Alice's operation to recover the original message. Because the output message $\rho_0$ is independent of the input $\rho$, this is secure against Eve, who does not know $i_0$.

Indeed, without further information, Eve's description of the situation is given by the random unitary channel $\Phi(\rho)=\sum_ip_iU_i \rho U_i^*=\rho_0$. In particular, if $\Phi$ maps distinct input states to the same state $\rho_0$, then Eve will not be able to distinguish between the states. This leads to the following formalization of the notion of private states.

\begin{definition}
Let $\mathcal S\subseteq \mathcal H_A$ be a set of pure states, let $\Phi:\mathcal L(\mathcal H_A)\rightarrow \mathcal L(\mathcal H_B)$ be a CPTP map, and let $\rho_0 \in \mathcal L(\mathcal H_B)$. Then $\mathcal S$ is \emph{private} for $\Phi$ with output state $\rho_0$ if 
\[
\Phi (\ket{\psi}\bra{\psi})=\rho_0 \quad \forall \ket{\psi}\in\mathcal S.
\]
If every state in the subspace spanned by  $\mathcal S$ has this property, then $\mathcal S$ is called a \emph{private subspace code} for $\Phi$.
\end{definition}

Often both the channel $\Phi$ itself, as well as the triple $[\m S, \Phi, \rho_0]$ from the above definition are called a  \emph{private quantum channel}. As in quantum error correction, it is desirable to ask that the private states in question form a subspace (i.e., a private subspace), so that arbitrary superpositions of states can be encoded as input states.

As an illustration, we discuss the single qubit case with the identity operator as output state. See \cite{CKPP11} for further details. Recall \cite{NC00} that any single qubit pure state $\ket{\psi}$ can be written as
\begin{eqnarray*}
\ket{\psi}=
\cos\frac\theta2 \begin{pmatrix} 1 \\ 0 \end{pmatrix}  +e^{i\varphi}\sin\frac\theta2 \begin{pmatrix} 0 \\ 1 \end{pmatrix}.
\end{eqnarray*}

We associate $\ket{\psi}$ with the point $(\theta, \varphi)$, in spherical coordinates, on the Bloch sphere (the unit 2-sphere) via $\alpha =\cos\left(\frac\theta2\right)$ and $\beta =e^{i\varphi}\sin\left(\frac\theta2\right)$. The associated Bloch vector is $\vec{r}=(\cos\varphi\sin\theta, \sin\varphi\sin\theta, \cos\theta)$. Using the Bloch sphere representation, we can associate
to any single qubit density operator $\rho$ a Bloch vector $\vec{r}\in\mathbb{R}^3$ satisfying $\|\vec{r}\|\leq{1}$, where
\begin{equation*}\label{eq:paulidecompofrho}
\rho=\frac{I+\vec{r}\cdot    \vec{\sigma} }2.
\end{equation*}
We use $\vec{\sigma}$ to denote the Pauli vector; that is, $\vec{\sigma}=(\sigma_x,\sigma_y, \sigma_z)^T$, where $\sigma_x,\sigma_y, \sigma_z$ are the Pauli $X, Y$, and $Z$ matrices, respectively.




Every unital qubit channel $\Phi$ can be represented as \cite{KR01} 
\begin{equation}\label{eq:channeldecomp}
	 \Phi\left(\frac1{2}\left[I+\vec{r}\cdot\vec{\sigma}\right]\right)=\frac1{2}\left[I+(T\vec{r})\cdot\vec{\sigma}\right],
\end{equation}
where $T$ is a $3 \times 3$ real matrix that represents a deformation of the Bloch sphere.

We are interested in cases where the private code $\mathcal S$ is nonempty. This is easily seen to occur precisely when $T$ in equation (\ref{eq:channeldecomp}) has non-trivial nullspace; we therefore consider the cases when the subspace of vectors $\vec{r}$ such that $T \vec{r} = 0$ is one, two, or three-dimensional. The following result summarizes the situation. Notice that case~(3) yields the only private subspace in the single qubit case. 

\begin{theorem}
Let $\Phi:\m L(\C^2\otimes \C^2)\rightarrow\m L(\C^2\otimes \C^2)$ be a unital qubit channel, with $T$ the mapping induced by $\Phi$ as in equation (\ref{eq:channeldecomp}). Then there are three possibilities for a private quantum channel $[\mathcal S, \Phi, \frac12 I]$ with $\mathcal S$ nonempty:
\begin{enumerate}
\item If the nullspace of $T$ is 1-dimensional, then $\mathcal S$ consists of a pair of orthonormal states.
\item If the nullspace of $T$ is 2-dimensional, then the set $\mathcal S$ is the set of all trace vectors (see below) of the subalgebra $U^*\Delta_2 U$ of $2\times 2$ diagonal matrices up to a unitary equivalence.
\item If the nullspace of $T$ is 3-dimensional, then $\Phi$ is the completely depolarizing channel (i.e., $\Phi(\rho) = \frac12 I$ $\forall \rho$) and $\mathcal S$ is the set of all unit vectors. In other words, $\mathcal S$ is the set of all trace vectors of $\mathbb{C}\cdot I_2$.
\end{enumerate}
\end{theorem}

Trace vectors are studied in the field of matrix theory, and were  initially introduced in the  work of Murray and von Neumann \cite{MvN37}. If  $\m A$ is a $*$-subalgebra of $\m L(\m H_n)$, then a vector $\ket{v}$ is a
\emph{trace vector} of $\m A$ if
\[
\bra{v}a\ket{v}=\frac1n\tr a \quad \forall a\in \m A.
\]

\section{Connection with Higher-Rank Numerical Ranges}

In the following result we derive a characterization of private quantum codes in terms of the dual map of a channel. 

			\begin{theorem}
\label{thm:KLAnalogue} Let $\Phi: \mathcal L(\mathcal H_A) \rightarrow \mathcal L(\mathcal H_B)$ be a CPTP map. Then a subspace $\m C$ of $\mathcal H_A$ is private for $\Phi$ with output state $\rho_0$; i.e., $\Phi(\rho) = \rho_0$ for all $\rho\in\mathcal L(\m C)$, if and only if for any $X \in \m L(\mathcal H_B)$  there exists a $\lambda_X\in \C$ such that
\[
P_{\m C} \Phi^\dagger (X) P_{\m C} = \lambda_{X} P_{\m C},
\]
where $P_{\m C}$ is the projection onto $\m C$. And in this case, $\lambda_{X} = \tr (\rho_0 X)$.
 \end{theorem}

\begin{proof}
Choose an orthonormal set $\{\ket{\phi_k}\}$ inside $\m H_A$ such that
$P_{\m C}=\sum_{k=1}^{\rank P_{\m C}}\ket{\phi_k}\bra{\phi_k}$. Let us first assume that $\Phi$ is private on $\m C$. By manipulating the bra-ket notation, and using the definition of $\Phi^\dagger$, we find
\begin{eqnarray*}
P_{\m C}\Phi^\dagger(X)P_{\m C}&=&\sum_{k,\ell=1}^{\rank P_{\m C}}\ket{\phi_k}\bra{\phi_k}\Phi^\dagger(X)\ket{\phi_\ell}\bra{\phi_\ell}\\
&=&\sum_{k,\ell}\bra{\phi_k}\Phi^\dagger(X)\ket{\phi_\ell}\ket{\phi_k}\bra{\phi_\ell}\\
&=&\sum_{k,\ell}\tr\left(\Phi^\dagger(X)\ket{\phi_\ell}\bra{\phi_k}\right)\ket{\phi_k}\bra{\phi_\ell}\\
&=&\sum_{k,\ell}\tr\left(X\,\Phi\left(\ket{\phi_\ell}\bra{\phi_k}\right)\right)\ket{\phi_k}\bra{\phi_\ell}\\
&=&\sum_{k,\ell}\tr\left(X\rho_0\right)\tr(\ket{\phi_k}\bra{\phi_\ell})\ket{\phi_k}\bra{\phi_\ell}\\
&=&\sum_{k}\tr\left(X\rho_0\right)\ket{\phi_k}\bra{\phi_k} = \tr\left(X\rho_0\right) P_{\m C},
\end{eqnarray*}
where the fifth equality follows because $\ket{\phi_\ell}\bra{\phi_k}\in \m L(\m C)$ and $\Phi$ is trace-preserving, so that $\Phi(\ket{\phi_\ell}\bra{\phi_k})=\tr(\ket{\phi_\ell}\bra{\phi_k})\rho_0$. 
Defining $\lambda_X:=\tr\left(X\rho_0\right)$ completes this direction of the proof.

For the other direction, we assume $P_{\m C}\Phi^\dagger(X)P_{\m C}=\lambda_XP_{\m C}$ for all $X \in \m L(\m H_B)$. Similar to the above calculation, we can write this as
\begin{eqnarray*}
\sum_{k}\tr\left(X\Phi\left(\ket{\phi_k}\bra{\phi_k}\right)\right)\ket{\phi_k}\bra{\phi_k} = \lambda_X\sum_{k}\ket{\phi_k}\bra{\phi_k}.
\end{eqnarray*}
For each $k$, compressing this equation by the rank one projection $\ket{\phi_k}\bra{\phi_k}$ yields
\[
\lambda_X=\tr\left(X\Phi\left(\ket{\phi_k}\bra{\phi_k}\right)\right)\quad \forall X \in \m L(\m H_B).
\]
The left hand side of the above equation does not depend on $k$, and so  $\Phi\left(\ket{\phi_k}\bra{\phi_k}\right)$ is a fixed density operator, say $\rho_0$, independent of $k$. And as the basis $\{ \ket{\phi_k} \}$ for $\m C$ was arbitrary it follows that $\Phi(\rho) = \rho_0$ for all $\rho\in \m L(\m C)$. Thus $\m C$ is private for $\Phi$ and this completes the proof.
\end{proof}
 
Note that if $\{\ket{e_i}\}$ is an orthonormal basis for $\m H_B$, then the scalars $\lambda_X$ obtained from the  matrix units  $\{\ket{e_i}\bra{e_j}\}$ form a Hermitian matrix; in particular the matrix coefficients for the output state $\rho_0$ in this basis. We also note that in many cases of interest, such as the class of Pauli channels for instance, the channel $\Phi$ is self-dual; that is, $\Phi = \Phi^\dagger$.

\subsection{Higher Rank Numerical Ranges}

The study of higher rank numerical ranges was initiated as an effort to help formalize the mathematical underpinnings for quantum error correction \cite{CKZ06a,CKZ06b}, and now the topic has taken on a life of its own with a rich collection of ongoing investigations. See \cite{CGHK08,GLPS11,LP11,LPS08,LPS09,Woe08}  as examples.

\begin{definition} The \emph{rank-$k$ numerical range} of an $n\times n$ matrix $X$ is
\[
\Lambda_k(X) = \{\lambda\in\C \,:\, PXP = \lambda P \textnormal{ for some  rank-$k$} \textnormal{ orthogonal projection } P\}.
\]
\end{definition}

The elements $\lambda\in \Lambda_k(X)$ are referred to as \emph{compression-values} for $X$, as they are obtained through compressions of $X$ to $k$-dimensional
subspaces. The case $k = 1$ reduces to the
numerical range $W(X)$ for operators (see, e.g.\ \cite{HJ91});
\[
\Lambda_1(X) =W(X) = \left\{\ip{X\psi}{\psi} \mid \ket{\psi}\textnormal{ any state in }\m H \right\}.
 \]
We refer to the sets $\Lambda_k(X)$,
$k > 1$, as the \emph{higher-rank numerical ranges} of $X$.

As discussed in \cite{CKZ06a,CKZ06b} and noted below, the
problem of finding correctable codes for a given quantum channel $\m  E (\rho) = \sum_i E_i \rho E_i^*$ is equivalent to finding the compression values inside the higher-rank
numerical ranges $\Lambda_k(E_i^* E_j)$ for all $i, j$ and all $k > 1$, along with the corresponding
projections.

\begin{remark}
Correspondingly, we make the following observation as a consequence of Theorem~\ref{thm:KLAnalogue}:
The problem of finding private codes for a given quantum channel $\Phi(\rho) = \sum_i V_i \rho V_i^*$ is equivalent to finding the compression values inside the higher-rank
numerical ranges $\Lambda_k(\Phi^\dagger(X))$ for all $X$ and all $k > 1$, along with the corresponding
projections. The output state $\rho_0$ is determined by its expectation values $\lambda_{X}=\tr (\rho_0 X)$. In other words, effectively the higher rank numerical ranges of images of a UCP map describe private quantum codes for the corresponding CPTP dual of the map. 
\end{remark}

\section{Stinespring Dilation Theorem \& Some Consequences}

        \subsection{Heisenberg \& Schr\"{o}dinger Dual Pictures}

Suppose that $\Phi: \mathcal B(\mathcal H_A)_t \rightarrow \mathcal B(\mathcal H_B)_t$ is a CPTP map.
Then there is a Hilbert space $\mathcal K$ (of dimension $K$ at most $\dim(\m H_A)\dim(\m H_B)$), a partial isometry $U\in \mathcal B(\mathcal H_A\otimes \mathcal K, \mathcal H_B \otimes \mathcal K)$, and a pure state $\ket{\psi}\in\mathcal K$ such that
\begin{eqnarray}\label{eqn:stinespring}
\Phi (\rho) = \tr_\mathcal K (U (\rho \otimes \ket{\psi}\bra{\psi})U^*) \quad\quad \forall \rho.
\end{eqnarray}
This is the Schr\"{o}dinger picture for the (discrete) time evolution of quantum states. If we define $V \ket{\phi} := U(\ket{\phi}\otimes\ket{\psi})$ for all pure states $\ket{\phi}\in\mathcal H_A$, for some fixed pure state $\ket{\psi}\in\mathcal K$, then we can write equation (\ref{eqn:stinespring}) more succinctly as $\Phi (\rho) = \tr_\mathcal K (V \rho V^*)$. In this way, the general form for $U$ is
\[
U = \left[ \begin{matrix} V& \vline & \ast \end{matrix} \right].
\]

The pair $(V, \m K)$ is called a Stinespring dilation of $\Phi$; assuming $\m K$ is minimal, $V$ is unique up to a unitary on $\mathcal K$. Here the Kraus operators for $\Phi$ can be read off as the coordinate operators of $V$; with 
$V =  \begin{pmatrix} V_1, \dots,  V_{K} \end{pmatrix}^T$,
where the Choi/Kraus decomposition of $\Phi$ is
\[
\Phi(X)=\sum_{i=1}^K V_i XV_i^*\quad \forall X, \textnormal{ where } V_i\in\mathcal B(\m H_A, \m H_B).
\]

A fundamental tool in quantum information -- the ``purification of mixed states'' -- can be viewed as a special case. Fix a density operator $\rho_0 \in \mathcal B(\mathcal H)_t$, and consider the CPTP map $\Phi : \mathbb{C} \rightarrow \mathcal B(\mathcal H)_t$ defined by
\[
\Phi (c \cdot 1) = c \, \rho_0 \quad \quad \forall c\in \mathbb{C}.
\]
Then the Stinespring Theorem gives (here $\mathcal K = \mathbb{C} \otimes \mathcal H = \mathcal H$):
\[
\rho_0= \Phi (1) = \tr_\mathcal K (U (1 \otimes \ket{\psi}\bra{\psi})U^*) = \tr_\mathcal K (\ket{\psi^\prime}\bra{\psi^\prime}),
\]
where $\ket{\psi'}\in\mathcal H\otimes\mathcal H$ is a ``purification'' of $\rho_0$ -- and the unitary freedom in the theorem captures all purifications. The idea that, given a mixed state in a Hilbert space, we can consider it to correspond to a pure state in a larger Hilbert space, is referred to as the \emph{Church of the Larger Hilbert Space} in quantum information theory.

The dual Heisenberg picture describes the time evolution of observables and is given as follows. Suppose that $\Phi^\dagger: \mathcal B(\mathcal H_B) \rightarrow \mathcal B(\mathcal H_A)$ is a UCP map.
    Then there is a Hilbert space $\mathcal K$ (of dimension $K$ at most $\dim(\m H_A)\dim(\m H_B)$) and an isometry $V\in \mathcal B(\mathcal H_A, \mathcal H_B\otimes \mathcal K)$ such that
\[
\Phi^\dagger (X) = V^*(X\otimes I_\mathcal K)V \quad\quad \forall X.
\]
%
	%
The $V$ above is the same $V$ as in the Schr\"{o}dinger picture.

	\subsection{Conjugate/Complementary Channels}

The Stinespring theorem naturally motivates the concept of conjugate, or complementary, quantum channels \cite{Hol06, KMNR07}.	
			\begin{definition}
Given a CPTP map $\Phi: \mathcal L(\mathcal H_A) \rightarrow \mathcal L(\mathcal H_B)$, consider a Stinespring representation $V\in \mathcal L(\mathcal H_A, \mathcal H_B\otimes \mathcal K)$ and $\mathcal K$ for which
\[
\Phi (\rho) = \tr_\mathcal K (V \rho V^*).
\]
Then the corresponding \textbf{conjugate} (or \textbf{complementary}) channel is the CPTP map $\widetilde{\Phi}: \mathcal L(\mathcal H_A) \rightarrow \mathcal L(\mathcal K)$ given by
\[
\widetilde{\Phi} (\rho) = \tr_\mathcal B (V \rho V^*).
\]
            \end{definition}
Any two conjugates $\widetilde{\Phi}$, $\Phi^\prime$ obtained in this way are related by a partial isometry $W$ such that $\widetilde{\Phi}(\cdot) = W \Phi^\prime (\cdot) W^*$. And so we talk of ``the'' conjugate channel for $\Phi$ with this understanding.

Kraus operators for conjugate channels can be computed directly from the original channel. Suppose that $V_i\in\mathcal L(\mathcal H_A,\mathcal H_B)$ are the Kraus operators for $\Phi: \mathcal L(\mathcal H_A) \rightarrow \mathcal L(\mathcal H_B)$; i.e., $\Phi(\rho) = \sum_i V_i \rho V_i^*$. Then we can obtain Kraus operators $\{ R_\mu \}$ for $\widetilde{\Phi}$ as follows:
       	
       \strut

Fix a basis $\{\ket{e_i}\}$ for $\mathcal K$ and define for $\rho\in\mathcal L(\mathcal H_A)$,
\[
F(\rho) = \sum_{i,j} \ket{e_i}\bra{e_j} \otimes V_i \rho V_j^* \in \mathcal L(\mathcal K \otimes \mathcal H_B).
\]

Then $\Phi(\rho) = \tr_\mathcal K F(\rho)$ and
\[
\widetilde{\Phi}(\rho) = \tr_\mathcal B F(\rho) = \sum_{i,j} \tr (V_i \rho V_j^*) \ket{e_i}\bra{e_j} = \sum_\mu R_\mu \rho R_\mu^*,
\]
where $R_\mu^* = [V_1^* \ket{f_\mu} \, V_2^* \ket{f_\mu} \cdots]$ and $\{ \ket{f_\mu}\}$ is an orthonormal basis for $\mathcal H_B$.

If we choose the standard bases for $\m K$ and $\m H_B$, then the $\mu$-th Kraus operator $R_\mu$ is obtained by simply ``stacking'',  top-down,  the $\mu$-th row of each Kraus operator $V_i, i=1, \dots, K$.

\subsection{Complementarity of Quantum Codes}

A subspace $\m C$ of $\m H$ is an \emph{error-correcting code} for a channel $\m E$ if there exists an error correction operation (a channel) $\m R$ such that
\begin{equation}\label{qec}
(\m R\circ \m E)(\ket{\psi}\bra{\psi})=\ket{\psi}\bra{\psi} \quad \forall \ket{\psi} \in \m C.
\end{equation}
By linearity, Eq.~(\ref{qec}) holds when $\ket{\psi}\bra{\psi}$ is replaced by an arbitrary density operator $\rho \in \m L(\m C)$. The following is a basic result that connects studies in quantum error correction and quantum cryptography.

			\begin{theorem}\cite{KKS08}\label{thm:KKS08} Given a conjugate pair of CPTP maps $\Phi$, $\widetilde{\Phi}$, a code is an error-correcting code for one if and only if it is a private code for the other.
            \end{theorem}

The extreme example of this phenomena is given by a unitary channel paired with the completely depolarizing channel---where the entire Hilbert space is the code. We illustrate this connection further with a pair of fairly simple examples.

		\begin{example}
Consider the 2-qubit swap channel $\Phi(\sigma\otimes\rho) = \rho \otimes \sigma$, which has a single Kraus operator, the swap unitary
\[
U=\begin{pmatrix} 1 & 0 & 0 & 0\\0 & 0 & 1 & 0\\ 0 & 1 & 0 & 0\\ 0 & 0& 0& 1 \end{pmatrix}. \]

\medskip
       	
The conjugate map $\widetilde{\Phi} : \mathcal L(\mathbb{C}^2\otimes\mathbb{C}^2) \rightarrow \mathbb{C}$ is implemented with four Kraus operators, which are
\begin{eqnarray*}
R_1 &=& \begin{pmatrix} 1 & 0 & 0 & 0  \end{pmatrix}
\quad
R_2 =  \begin{pmatrix} 0 & 0 & 1 & 0  \end{pmatrix}
\\
R_3 &=& \begin{pmatrix} 0 & 1 & 0 & 0  \end{pmatrix}
\quad
R_4 =  \begin{pmatrix} 0 & 0 & 0 & 1  \end{pmatrix},
\end{eqnarray*}
and one can easily see that $\widetilde{\Phi}(\rho)= \tr \rho= 1$ for all $\rho\in\mathcal L(\mathbb{C}^2\otimes \mathbb{C}^2)$.

\end{example}

\begin{example}			
			Consider the 2-qubit phase flip channel $\Phi$ with (equally weighted) Kraus operators $\{I,Z_1\}$, where $Z_1$ is shorthand for $Z\otimes I_2$. The dilation Hilbert space here is 3-qubits in size, and the conjugate channel $\widetilde{\Phi} : \mathcal L(\mathbb{C}^2\otimes\mathbb{C}^2) \rightarrow \mathcal L(\mathbb{C}^2)$ is implemented with the following Kraus operators:

\begin{eqnarray*}
R_1 &=& \frac{1}{\sqrt{2}} \begin{pmatrix} 1 & 0 & 0 & 0 \\ 1 & 0 & 0 & 0  \end{pmatrix}
\quad
R_2 = \frac{1}{\sqrt{2}}\begin{pmatrix} 0 & 1 & 0 & 0 \\ 0 & 1 & 0 & 0  \end{pmatrix}
\\
R_3 &=& \frac{1}{\sqrt{2}} \begin{pmatrix} 0 & 0 & 1 & 0 \\ 0 & 0 & -1 & 0 \end{pmatrix}
\quad
R_4 = \frac{1}{\sqrt{2}}\begin{pmatrix} 0 & 0 & 0 & 1 \\ 0 & 0 & 0 & -1 \end{pmatrix}.
\end{eqnarray*}

 We have $\Phi(\rho) = \frac{1}{2}(\rho + Z_1 \rho Z_1)$ and $\widetilde{\Phi} (\rho) = \sum_{i=1}^4 R_i \rho R_i^*$ for all 2-qubit $\rho$.  \medskip
 It is clear that the code $\{ \ket{00}, \ket{01} \}$ is correctable for $\Phi$; in fact it is noiseless/decoherence-free. And thus we know it is private for the conjugate channel $\widetilde{\Phi}$.\medskip

 Indeed, one can check directly that every density operator $\rho$ supported on $\{ \ket{00}, \ket{01} \}$, satisfies
 \begin{eqnarray*}
 \widetilde{\Phi} (\rho) &=& \ket{+}\bra{+} = \frac{1}{2} (\ket{0} + \ket{1})(\bra{0} + \bra{1})\\
 &=&\frac12\begin{pmatrix} 1 &1 \\ 1 & 1 \end{pmatrix}.
\end{eqnarray*}
\end{example}

\subsection{Alternate Proof of Theorem~\ref{thm:KLAnalogue}}

The celebrated Knill-Laflamme theorem in quantum error correction \cite{KL97} gives testable conditions in terms of Kraus operators for determining whether a given code is correctable. Here, $\m C$ denotes a subspace of a Hilbert space and $P_{\m C}$ is the projection onto $\m C$. If $\m E$ is a channel with Kraus operators $\{E_i\}$, then $\m C$ is correctable for $\m E$ if and only if
\begin{equation}\label{KLcond}
P_{\m C}E_i^* E_jP_{\m C}=\lambda_{ij}P_{\m C},
\end{equation}
for some Hermitian matrix $\lambda=[\lambda_{ij}]_{ij}$ of complex numbers.

With both the conjugate channel machinery and Knill-Laflamme result in hand, we can arrive at the conclusion of Theorem~\ref{thm:KLAnalogue} from an alternate perspective.  Indeed, letting  $\{V_i\}_{i=1}^K$ be the Kraus operators of the channel $\Phi$ and $\{R_\mu\}_{\mu=1}^b$ be the Kraus operators of its conjugate channel $\tilde{\Phi}$ (as before),  and $\{\ket{f_\mu}\}$ any orthonormal basis for $\m H_B$, we compute
\begin{eqnarray}\label{eq:KrausConjDual}
R_\mu^* R_\nu=
\begin{pmatrix}
V_1^* \ket{f_\mu} &V_2^* \ket{f_\mu} &\dots &V_K^* \ket{f_\mu}
\end{pmatrix}
\begin{pmatrix}
\bra{f_\nu}V_1\\
\bra{f_\nu}V_2\\
\vdots\\
\bra{f_\nu}V_K
\end{pmatrix}=\sum_{i=1}^K V_i^*( \ket{f_\mu} \bra{f_\nu})V_i.  
\end{eqnarray}

We recall the fact that the code being private for $\Phi$ is equivalent to it being correctable for the conjugate channel $\tilde{\Phi}$.
Then, by the Knill-Laflamme conditions Eq.~(\ref{KLcond}) and the calculation of Eq.~(\ref{eq:KrausConjDual}), we have for all $\mu, \nu\in \{1,\dots, b\}$ that $\m C$ is private for $\Phi$ if and only if
\begin{eqnarray*}
\lambda_{\mu \nu}P_{\m C}&=&P_{\m C}R_\mu^* R_\nu P_{\m C}  
= P_{\m C}\sum_{i=1}^K V_i^*( \ket{f_\mu} \bra{f_\nu})V_iP_{\m C} 
= P_{\m C}\Phi^\dagger( \ket{f_\mu} \bra{f_\nu})P_{\m C}.
\end{eqnarray*}
And as the identity holds for arbitrary matrix units $\ket{f_\mu} \bra{f_\nu}$, by linearity it extends to all operators $X$ as in Theorem~\ref{thm:KLAnalogue}.

\section{Outlook}
We have given an introduction to private quantum codes and connected them with the study of higher rank numerical ranges. On the other hand, higher rank numerical ranges have proved to be interesting objects of study in their own right. They have also been a useful tool for studying quantum error correcting codes, and this work suggests the same could be true in the study of private codes. For brevity we have focused on private subspaces here, but as in the case of quantum error correction there is also a notion of private subsystems that can be connected with higher rank numerical ranges. All of these topics warrant further investigation. 

\vspace{0.1in}

{\noindent}\emph{Acknowledgements.} D.W.K. was supported by NSERC Discovery Grant 400160 and NSERC Discovery Accelerator Supplement 400233. S.P. was supported by an NSERC Doctoral Scholarship.

\label{lastpage}

\end{document}